\documentclass[article,11pt]{amsart}
\usepackage{latexsym,graphicx,verbatim}
\usepackage{amsmath,amssymb,cite,color,xcolor}
\usepackage{latexsym,amsmath,amsfonts,amssymb}
\usepackage{color,latexsym,amsmath,amsthm,amsfonts,amssymb,xcolor}
\usepackage[utf8]{inputenc}

\author[Candeloro]{\bf Domenico Candeloro}
\address{Department of Mathematics and Computer Sciences \\ University of Perugia\\
Via Vanvitelli, 1 - 06123 Perugia (Italy), Orcid Id:  0000-0003-0526-5334}
\email{candelor@dmi.unipg.it}

\author[Sambucini]{\bf Anna  Rita Sambucini}
\address{\rm Department of Mathematics and Computer Sciences \\ University of Perugia\\ Via Vanvitelli, 1 - 06123 Perugia (Italy) , Orcid Id: 0000-0003-0161-8729} \email{anna.sambucini@unipg.it}

\author[Trastulli]{\bf Luca Trastulli}
\address{\rm Department of Mathematics and Computer Sciences \\ University of Perugia\\ Via Vanvitelli, 1 - 06123 Perugia (Italy) , Orcid Id:  0000-0002-7722-4008} 
\email{luca.trastulli@gmail.com}
\subjclass[2010]{28B20, 28B05,28B05, 26E25, 46B20, 54C60.}

\keywords{Pettis multivalued integral,  martingale, Girsanov Theorem.}

\title{A Girsanov result for the Pettis  integral}
 \newtheorem{theorem}{Theorem}[section]
  \newtheorem{definition}[theorem]{Definition}
 
  \newtheorem{proposition}[theorem]{Proposition}
 
 \newtheorem{rem}[theorem]{Remark}

\newcommand{\spaces}{$\left(\Omega, \mathcal{A},\mathbb{P},\mathcal{F}\right)$ }

\newcommand{\dint}{\displaystyle{\int}}

\begin{document}
\maketitle
\begin{abstract}
A kind of Pettis integral representation for a Banach valued  Itô process is given and its drift term is modified using a Girsanov Theorem. 
\end{abstract}

\section{Introduction}
A very important tool in    measure theory and in  mathematical finance  is the Girsanov theorem, strictly linked to the well-known Wiener stochastic process called the standard
 Brownian motion $(w_t)_{t\in[0,\infty)}$, defined on a probability space $\left(\Omega, \mathcal{A}, \mathbb{P}\right)$  (a classic formulation of this result in the scalar case
  can be found for example in \cite{mikosh}, while for extension to vector lattices see \cite{GL}). This theorem allows to change the  probability measure $\mathbb{P}$, through the definition of a Radon-Nikod\'ym derivative,
   in order to obtain an 
equivalent measure $\mathbb{Q}$ such that, if $w_t$ is the standard Brownian motion on the probability space $\left(\Omega, \mathcal{A}, \mathbb{P}\right)$ (and then it 
results to be a martingale in itself under $\mathbb{P}$),   its transform $\tilde{w}_t$ is still a Brownian motion on the probability space $\left(\Omega, \mathcal{A}, \mathbb{Q}\right)$. 
The new measure $\mathbb{Q}$ is called an equivalent martingale measure  for $w_t$ with respect to $\mathbb{P}$.
The necessity of changing measure arises, for example, 
 in the Black-Scholes models (in this context they are called neutral risk measure).
So Girsanov Theorem describes how the dynamics of stochastic processes change when the original measure is changed to an equivalent probability measure.
 At the same time, for many applications, such as conditional measures we need to work with measures or random variables taking values in a suitable Banach space. 
In this paper we want to  generalize the Girsanov Theorem
in a more abstract contest.

This research could be motivated, for example, by  the study of  
a Brownian motion $w_t$, conditioned by the future $w_T$. 
We follow the idea formulated in \cite{mikosh} for the real Brownian motion, that defines the Radon-Nikod\'ym derivative using the density functions of the processes $w_t$ and of its transform $\tilde{w}_t:=w_{t}+\dint_0^t r(s)ds$, for a suitable scalar function $r$  that links the drift and diffusion terms.
Then it follows that, under the new measure $\mathbb{Q}$, the transformed process $\tilde{w}_t$ is a Brownian motion and a martingale in itself.

In Banach spaces were introduced and studied 
different types of integrals that generalize the Bochner one. We want to point out that this topic is  interesting also from the point of view of measure and integration theory, 
as showed in the papers \cite{mar1,candeloro0,candeloro1,BS1,
ckr,ckr1,LM1,dm,dma,
Fremlin0,musial,pettis,T84,rodriguez1,dp,bc,brow,anca,mu,mu2011,dps1,dm-new}.\\
In a previous research we have obtained a Girsanov result for the Birkhoff  integral  of a vector-valued function \cite{girsanov,cst1} and a Radon Nikodym result in \cite{candrn}.  
In this paper we want to weaken the hypothesis of integrability and examine the case of the Pettis integrability  when the space $X$ is a  Banach space not necessarily separable, 
the use of non-absolute integrals is also motivated by applications, as shown in \cite{cc, dps1,dmas}. 

The organization of the paper is as follows:  in Section \ref{defin}  we will introduce the Pettis stochastic integral and we recall some results, 
 while  in Section \ref{main}  a  Girsanov result (Theorem \ref{coroll Girsanov Pettis}) for vector measures is obtained.
Finally  a particular case is investigated. 

\section{Definitions}\label{defin}
We recall some definitions. Let $I\subset \mathbb{R}$ be an interval of the real line,\ $\left(\Omega, \mathcal{A}\right)$ a measurable space, 
 $\nu: \mathcal{A} \to \mathbb{R}_0^+$ be a scalar measure. With the symbol $\mathbb{P}$
($\mathbb{P}: \mathcal{A} \to [0,1]$)  we denote a probability measure.
\begin{definition}\rm
Given two measures $\mathbb{P}, \mathbb{Q}$ on a measurable space $\left(\Omega, \mathcal{A}\right)$, we say that $\mathbb{Q}$ is {\em absolutely continuous}
 on $\mathcal{A}$ with respect to $\mathbb{P}$,  ($\mathbb{Q}\ll\mathbb{P}$), if for every $A\in \mathcal{A}$ such that $\mathbb{P}(A)=0$, we have
  that $\mathbb{Q}(A)=0$. Two measures that are absolutely continuous one respect to the other are said to be two {\em equivalent measures}.
 \end{definition} 
Let $(X,\left\|\cdot \right\|)$ be a Banach space, with  the topology induced by its norm, $X^{\ast}$ its dual  and $\mathcal{B}(X)$ be the $\sigma$-algebra
 of Borel on $X$.
\\

A {\em filtration} $\mathcal{F}=\left(\mathcal{F}_t\right)_{t\in I}$ on a probability space $\left(\Omega, \mathcal{A},\mathbb{P}\right)$ is an increasing 
family of sub-$\sigma$-algebras of $\mathcal{A}$.
 $\left(\Omega, \mathcal{A},\mathbb{P}\right)$ provided with a filtration $\mathcal{F}_t$ is said to be
 a {\em filtered space}.
\begin{definition}\rm
We say that a collection of random variables $\left(Z_t\right)_{t\in I}$ is an {\em X-valued stochastic process} if the function $Z:(I\times \Omega,\mathcal{B}(I)\otimes \mathcal{A})\rightarrow (X,\mathcal{B}(X))$ is a measurable function.
\end{definition}
If $Z_t$ is an $X$-valued stochastic process, then 
\begin{itemize}
\item for every $t\in I$, $Z_t$ is a random variable that takes values in $X$, while
\item  the function $Z(\omega): I\rightarrow X$ for every $\omega\in \Omega$ is called a {\em trajectory} of the process $Z$;
\item $Z$ is {\em continuous a.s} if for almost all $\omega\in \Omega$ the trajectory $Z(\omega,\cdot)$ is a continuous function from $I$ to $X$.
\end{itemize}

\begin{definition} \rm 
Given a stochastic process $Z = (Z_t)_{t\in [0,T]}$ , 
\begin{itemize}
\item the {\em natural filtration}
for $Z$ is defined for every $t$ as the $\sigma$-algebra described by the stochastic process at all the previous times, namely 
$\tilde{\mathcal{F}_t}=\sigma\left\{ Z_s,  \quad s\in [0,t]\right\};$
\item 
 $Z$ is {\em adapted} to a filtration $\left(\mathcal{F}_t\right)_t$ if $Z_t$ is $\mathcal{F}_t$ measurable for every $t\in [0,T]$.
Obviously, every process is adapted to its natural filtration.
\item
 $Z$,   it is {\em progressively measurable} with respect to the filtration $\mathcal{F}_t$, if for every $t\in I$ we have that $Z|_{[0,t]\times \Omega}$ is $\mathcal{B}([0,t])\times \mathcal{F}_t$ measurable.
\end{itemize}
\end{definition}
We use capital letters as $Z_t$ to refer to $X$-valued stochastic processes, while we use lower case letters as $z_t$ to refer to real stochastic processes.
\begin{definition}\label{scalare-mart} \rm
 We say that $z_t$ is a {\em martingale} with respect to the filtration $\left(\mathcal{F}_t\right)_t$ if $z_t$ is $\mathcal{F}_t$ adapted and $\mathbb{E}(z_t|\mathcal{F}_s)=z_s$, 
 for all $t,s\in I$, $s\leq t$. A process that is a martingale with respect to its natural filtration is said to be a {\em martingale in itself.}
 \end{definition}
 It is important to notice that in the definition of martingale the underlying measure has a key role in the expected value, so a process could be or not a martingale depending on the measure $\mathbb{P}$ that we consider. 
If the process $z_t$ is a martingale (under $\mathbb{P}$), then the measure $\mathbb{P}$ is called a {\em martingale measure}.\\
 An equivalent martingale measure is also called a risk-neutral measure and it is used for example  in  financial market in order to obtain  an arbitrage-free price for each asset in the market.\\
  
Sometimes, it is useful to work with stochastic processes that satisfy a local version of  the martingale  property. 
\begin{definition} \rm 
A real stochastic process $z_t$, $t\in [0,T]$ defined on a filtered probability space $\left(\Omega, \mathcal{A},\mathbb{P},\mathcal{F}\right)$ is a {\em local martingale} with
 respect to $\mathcal{F}_t$ if there exists an increasing sequence $(\tau_n)_n$ of $\mathcal{F}_t$-stopping times, such that $\displaystyle\lim_{n \rightarrow\infty}\tau_n=T$ 
 and for every $n\in \mathbb{N}$, the process $z_{t\wedge \tau_n}$ is a $\mathcal{F}$-martingale.
\end{definition}
Among real stochastic processes we consider
\begin{definition} \rm 
 $(w_t)_{t\in [0,T]}$ is a real Brownian motion if it satisfies these properties:
\begin{itemize}
\item $w_0=0$ a.s. ;
\item $w_t$ is $\mathcal{F}_t$ adapted and continuous a.s. ;
\item $w_t$ has independent increments i.e. for very $s,t\in [0,T]$ such that $s\leq t$ we have that the increment $w_t-w_s$ is independent by $\mathcal{F}_s$ and it follows a normal distribution of parameters $\mathcal{N}\left(0,t-s\right)$.
\end{itemize}
\end{definition} 
\begin{definition}
Given $p\geq1$, a stochastic process $z_t$ belongs to the class $\mathbb{L}^p(\mathcal{F})$ if  $z_t$ is a progressively measurable process with respect to the filtration $\mathcal{F}_t$ and
$\displaystyle\int_0^T\mathbb{E}(|z_t|^p)dt<\infty$.
\end{definition}

 \begin{definition}\label{exp_Martingale}Given the standard Brownian motion $\left(w_t\right)_{t\in [0,T]}$ on the filtered space of probability \spaces , where the filtration is the natural one of the Brownian motion and a process $\theta_t \in \mathbb{L}^2(\mathcal{F})$, we define the {\textit{exponential martingale}} relative to $\theta$ as 

\begin{equation}y_t(\theta)=\exp\left\lbrace -\dfrac{1}{2}\int_0^t\theta_s^2ds-\int_0^t\theta_sdw_s\right\rbrace.
\label{1.1}
\end{equation}
\end{definition}

From now on we denote with $\left(w_t\right)_{t\in [0,T]}$ the standard real Brownian motion on the probability space $\left(\Omega, \mathcal{A},\mathbb{P}\right),$ and, unless otherwise specified, we denote with $\mathcal{F}_t$ its natural filtration. We recall that, by independent increments property, the Brownian motion $w_t$ is a continuous, Gaussian
 real process and has a normal distribution of parameters $\mathcal{N}\left(0,t\right)$, it results to be a  martingale in itself and it is progressively measurable with respect to its
  natural filtration.
  \\
  For what is unexplained relatively to stochastic processes we refer to \cite{mikosh,bc,neerven,novikov,RvG,KLW}.
\begin{definition} \rm
We say that $\Phi:\Omega\rightarrow X$ is  {\em Pettis integrable} with respect to $\nu$  if $\Phi$ is weakly measurable and for every set $D\in \mathcal{A}$ there exists an element $x_D\in X$, such that
\[x^{\ast}(x_D)=\int_Dx^{\ast}(\Phi)d\nu.\]
We say that $x_D$ is the {\em Pettis integral} of the function $\Phi$ over $D$ and we write
$x_D:=(Pe)\dint_D\Phi d\nu$ and we use the symbol $Pe(\Omega, \nu)$ to denote the class of Pettis integrable functions $\Phi$ with respect to $\nu$.
\end{definition}
For more details on the Pettis integral and its properties  we refer to \cite{ckr,ckr1, pettis,mu,musial,BS1
}.
 Together 
 to the Pettis integral of a vector function $\Phi$ with respect to a scalar measure $\nu$ we also consider the integral of a scalar function $\phi$ with respect to a vector measure
$N$ in the sense of Bartle-Dunford-Schwartz, namely
\begin{definition} \rm
We say that a measurable $\phi:\Omega\rightarrow \mathbb{R}$ is  {\em Bartle-Dunford-Schwartz integrable}
  with respect to  $N: \mathcal{A} \to X$ if for every  $E\in\mathcal{A}$ there exists $M(E) \in X$ such that
\[ (x^*, M(E) ) = \int_E \phi \, d(x^*,N), \qquad \forall \, x^* \in X^*.\]
Then 
$
(BDS)\dint_E \phi\,dN:=M(E) 
$ and we denote by $BDS (\Omega, N)$ the space of all Bartle-Dunford-Schwartz integrable functions. 
\end{definition}
For properties and details of this type of integrations consider for example \cite{stef,FMN,mmodena}.
In particular it is known that 
\begin{proposition}\label{stefansson} {\rm (\cite[Proposition 8]{stef})}
Given $N(\cdot) := (Pe)\dint_{\hskip-.2cm\cdot} \, \Phi d\nu$, 
a measurable scalar function $\phi \in BDS(\Omega, N)$ if and only if $\phi \Phi \in Pe(\Omega, \nu)$ and 
\[ \int \phi\,  dN = \int \phi \, \Phi \, d\nu.\]
\end{proposition}

\begin{definition} \rm 
Let $\Phi\in Pe\left(\Omega,\nu\right)$ and $\mathcal{F} \subset \mathcal{A}$ be a $\sigma$-algebra. We define, provided that it exists, the \emph{Conditional expectation of $\Phi$
 with respect to $\mathcal{F}$}, indicated by $\mathbb{E}\left(\Phi|\mathcal{F}\right)$, as the weakly  $\mathcal{F}$-measurable function $\Psi$ such that  $\Psi\in Pe(\Omega,\nu)$
  and for every $E\in \mathcal{F}$ it holds 
$$(Pe)\int_E \Phi d\nu=(Pe)\int_E\Psi d\nu.$$
\end{definition}
From this definition the classical tower property follows together with:
\begin{proposition}\label{tre}
Let $\Phi:\Omega\rightarrow X$ be a vector valued function and $\mathcal{F}\subset \mathcal{A}$ such that  there exists the conditional expectation 
$\mathbb{E}\left(\Phi|\mathcal{F}\right).$ Then, given a scalar $\mathcal{F}$ measurable function $\phi:\Omega\rightarrow \mathbb{R}$ 
so that the product function $\Phi(\cdot)\phi(\cdot)$ is Pettis integrable with respect to $\nu$, it is:
$$\mathbb{E}\left(\Phi(\omega)\phi(\omega)|\mathcal{F}\right)=\phi(\omega)\mathbb{E}\left(\Phi|\mathcal{F}\right).$$
\end{proposition}
\begin{proof}
 Taking into account Proposition \ref{stefansson},
the proof is analogous to that of  \cite[Theorem 2.13]{cst1}
 where the Birkhoff integrability of first type of $\Phi(\cdot)\phi(\cdot)$ is substituted by its Pettis integrability and $\phi$ is  Bartle-Dunford-Schwartz integrable.
\end{proof}

\section{A Girsanov result}\label{main}

Now we define the stochastic integral of $\Phi$ with respect to a Brownian motion $(w_t)_t$ as follows.

\begin{definition}\label{stochastic pettis integral} \rm
Let $\Phi:[0,T]\rightarrow X$ be a vector function and $w_t$ the standard Brownian motion on the filtered space \spaces . Suppose that the function
 $\Phi$ is  weakly 
measurable and the function $x^{\ast}\left(\Phi\right)$ belongs to $\mathbb{L}^2(\mathcal{F})$, for every $x^{\ast}\in X^{\ast}$. If for every 
$B\in\mathcal{B}\left([0,T]\right)$ there exists an $X$-valued random variable $Y_B:\Omega\rightarrow X$ such that 
\begin{equation}\label{stoch integral}
x^{\ast}(Y_B)=\int_Bx^{\ast}(\Phi(s))dw_s
\end{equation}
then we say that $\Phi$ is stochastically integrable with respect to $w_t$ and we write
$$Y_B=(Pe)\int_B\Phi(s)dw_s$$
\end{definition}
We refer to this integral as a Stochastic Pettis integral, with respect to the Brownian motion. 
This integral is defined as a Pettis integral, i.e. in a weak sense but is a stochastic integral. \\


Moreover
\begin{theorem}\label{martingala pettis}
Let $\Phi:[0,T]\rightarrow X$ be a vector function, stochastically integrable with respect to $w_t$, the standard Brownian motion as
 in (\ref{stoch integral}). Then the process defined by
$$ \left((Pe)\int_0^t\Phi(s)dw_s\right)_t$$
is an $\left(\mathcal{F}_t\right)_t$- martingale. 
\end{theorem}
\begin{proof}
Let $s\leq t$ in $[0,T]$ and fix $x^{\ast}\in X^{\ast}$. We have that
\begin{eqnarray*}
<x^{\ast},\mathbb{E}\left((Pe)\int_0^t\Phi(r)dw_r|\mathcal{F}_s\right)>&=&
\mathbb{E}\left(\int_0^t<x^{\ast},\Phi(r)>dw_r|\mathcal{F}_s\right)= \\&=&
\mathbb{E}\left(\int_0^s<x^{\ast},\Phi(r)>dw_r\right)=\\
&=&<x^{\ast},(Pe)\int_0^s\Phi(r)dw_r>.
\end{eqnarray*}
Since  $(\dint_0^t<x^{\ast},\Phi(r)>dw_r)_t$  is $\left(\mathcal{F}_t\right)_t$-adapted and in  $\mathbb{L}^2(\mathcal{F})$, 
it is a  scalar martingale with respect to the natural filtration of the Brownian motion (Definition \ref{scalare-mart}).  
So, by arbitrariness of $x^{\ast}$, we have
$$\mathbb{E}\left((Pe)\int_0^t\Phi(r)dw_r|\mathcal{F}_s\right)=(Pe)\int_0^s\Phi(r)dw_r.$$
\end{proof}

The class of Itô processes is  very important and has lots of applications in mathematics. The idea is to make some kind of Pettis integral representation for a Banach valued process.
 In fact,  using the Girsanov Theorem, it is possible to change the drift term of an Itô integral and obtain a local martingale. We want to prove an analogous result for vector
  processes that have an integral representation in terms of Pettis integrals.
The main problem is to define a stochastic integral, that is an integral of a function with respect to a stochastic process. In the real case in fact we have
$$x_t=\int_o^ta_sds+\int_0^tb_sdw_s$$
and we want to extend the term $\dint_0^tb_sdw_s$ from the real case to the vector one.
Firstly, we could consider this kind of definition of stochastic integral for Banach valued functions.
\begin{definition} \rm
Let's consider a Banach valued  process $A_t:[0,T]\times \Omega \rightarrow X$ 
which admits a stochastic integral representation 
of the following  form (Itô-Pettis Process)
$$ A_t=(Pe)\int_0^t\Psi(s)ds+(Pe)\int_0^t\Phi(s)dw_s, \ \ \ \ \ \ \left(\text{under} \ \ \ \mathbb{P}\right).$$
on the probability filtered space \spaces , where $\Psi$ and $\Phi$ are two Pettis integrable functions, taking values in 
$X$ and $\Phi$ is Pettis-stochastic integrable with respect to the Brownian motion $w_t$.
The component $\Psi(t)$ is the {\em drift term } of the process $A_t$, while $\Phi(t)$ represents its {\em diffusion term}.
\end{definition}

The idea, following the Girsanov construction, is to change the drift term into the expression of the Itô process, using a change of probability measure. This can be obtained defining a new measure by a martingale positive process that generally takes the form of an exponential. This family of martingales takes the name of {\em exponential martingales}. 

\begin{theorem}\label{coroll Girsanov Pettis} {\rm ([a Girsanov result])}
If the following conditions are satisfied:
\begin{enumerate}
\item There exists a scalar function $r:[0,T]\rightarrow \mathbb{R}$ in $L^2([0,T])$ such that, for every $t\in [0,T]$ one has
$\Phi(t)=r(t)\Psi(t).$
\item The exponential process 
$y_t=\exp\left\lbrace-\dfrac{1}{2} \displaystyle{\int_0^T} r^2(s)ds- \displaystyle{\int_0^T} r(s)dw_s\right\rbrace$ is a martingale
 under $\mathbb{P}$ with respect to $\mathcal{F}$. 
\end{enumerate}
Then, defined 
$\tilde{w_t}=w_{t}+\displaystyle{\int_0^t} r(s)ds,$
we have that
$$A_t=(Pe)\int_0^t\Psi(s)d\tilde{w}_s, \qquad \left(\text{under} \ \ \mathbb{Q}\right).$$ 
where
$\dfrac{d\mathbb{Q}}{d\mathbb{P}}=y_T.$
Therefore, the process $A_t$ is a martingale under $\mathbb{Q}$ and then $\mathbb{Q}$ is an equivalent martingale measure with respect to $\mathbb{P}$.
\end{theorem}

\begin{proof}
Since  $y_t$ is an exponential martingale (as in Definition \ref{exp_Martingale}), then thanks to
\cite[Theorem 10.5]{pascucci}, 
 the measure $\mathbb{Q} = \dint y_T d\mathbb{P} $ defined using $Y_{T}$ as Radon-Nikodym derivative, is an equivalent martingale measure with respect to $\mathbb{P}$
 and the process $\tilde{w_t}$ is a
 Brownian motion under the new probability  $\mathbb{Q}$. This allows us to define the stochastic Pettis integral
$$(Pe)\int_0^t\Psi(s)d\tilde{w}_s \qquad \left(\text{under} \ \ \mathbb{Q}\right).$$
Now we claim that
\begin{equation}\label{eq pettis}
(Pe)\int_0^t\Psi(s)d\tilde{w}_s=(Pe)\int_0^t\Psi(s)dw_s
 +
(Pe)\int_0^tr(s)\Psi(s)ds
\end{equation}
as a vector equivalence of Pettis stochastic integral and Pettis integral.
To prove this we consider, for every $x^{\ast}\in X^{\ast}$
\begin{eqnarray*}
<x^{\ast},\int_0^t\Psi(s)d\tilde{w}_s>&=&\int_0^t<x^{\ast},\Psi(s)>d\tilde{w}_s=\int_0^t<x^{\ast},\Psi(s)>\left(dw_s
+
r(s)ds\right)=\\
&=&\int_0^t<x^{\ast},\Psi(s)>dw_s
+
\int_0^t<x^{\ast},\Psi(s)>r(s)ds=\\
&=&<x^{\ast},\int_0^t\Psi(s)dw_s>
+
<x^{\ast},\int_0^t\Psi(s)r(s)ds>=\\
&=&<x^{\ast},\int_0^t\Psi(s)dw_s
+
\int_0^t\Psi(s)r(s)ds>,
\end{eqnarray*}
where all the vector integrals are Pettis stochastic integrals.
So by the arbitrariness of $x^{\ast}\in X^{\ast}$, the equation (\ref{eq pettis}) holds. Then, observing that $d\tilde{w}_t=dw_t+r(t)dt$, we easily deduce that
\begin{eqnarray*}
A_t \hskip-.2cm&=&\hskip-.2cm(Pe)\int_0^t\Psi(s)ds+(Pe)\int_0^t\Phi(s)dw_s
=(Pe)\int_0^t\Psi(s)ds+(Pe)\int_0^t\Phi(s)d\tilde{w}_s+\\
&-&(Pe)\int_0^tr(s)\Phi(s)ds=(Pe)\int_0^t\Phi(s)d\tilde{w}_s
\end{eqnarray*}
and then we have that, under the equivalent measure $\mathbb{Q}$
$$A_t=(Pe)\int_0^t\Psi(s)d\tilde{w}_s$$ and it turns out to be a martingale thanks to Theorem \ref{martingala pettis}.
\end{proof}

The conditioning of random variables to future time (past times) is very useful in some application in mathematical finance and for pricing formulas. 
We can see conditional measures as vector-valued measures and, using the Pettis integral, we can give an example of application of  Theorem \ref{coroll Girsanov Pettis}
 for conditional measures.
 However if we condition a Brownian motion on an expiration time $T>0$ fixed, the distribution of  this process changes and in general, it doesn't preserve some of its properties,  such as the martingale property. 

\begin{rem}\label{ConditionalDistribution} \rm  
 It's well know that the Brownian motion $w_t$, conditioned by the future $w_T$, has a conditioned density function, 
 that could be seen like a vector function from the real line to $L^1\left(\Omega\right)$,
  given by
$$ f_{t}(x|w_{T}) :=\dfrac{\sqrt{T}}{\sqrt{T-t}\sqrt{2\pi t}}\exp\left\lbrace-\frac{\left( x-\frac{w_{T}t}{T}\right)^{2}T}{2(T-t)t}\right\rbrace.$$
For details about conditional distributions for Gaussian and Wiener processes, we refer to \cite{ex10}.
In particular, we recall that, if $\left(z_{t}\right)_{t}$ is a Gaussian process, with parameters $\mathcal{N}\left(\mu_{t},\sigma^{2}_{t}\right)$, than, given $u\ge t$, we have that 
\begin{equation}\label{conditionalexpectation}
\mathbb{E}\left(z_{t}|z_{u}\right)=\mu_{t}+\rho\dfrac{\sigma_{t}}{\sigma_{u}}\left(z_{u}-\mu_{u}\right) 
\end{equation}
\end{rem}

Now we consider  the expected value of $w_t$ that under the conditioned measure, namely
\begin{proposition}
Given a measurable function $\phi: \Omega \to \mathbb{R}$ ,
if  $\phi \in  BDS(\Omega, N)$  then 
$$\int_A\phi(\omega)dN(\omega)=
\int_A \mathbb{E}\left(\phi(\omega)|w_T(\omega)\right)d\mathbb{P}(\omega)$$
where
\begin{eqnarray}\label{due}
N\left(A\right):=
\dint_A\mathbb{E}\left(\cdot |w_T(\omega)\right)d\mathbb{P}(\omega).
\end{eqnarray}
\end{proposition}

\begin{proof}
Recalling the conditional distribution of Gaussian processes and conditionated to a future time  (see for example \cite[Section 4.6]{ex10}   and formula (\ref{conditionalexpectation})), we have that $w_{t}|w_{T}$ follows a Gaussian distribution, as seen in Remark \ref{ConditionalDistribution}, with expected value, with respect the probability measure $\mathbb{P}$, given by
\begin{eqnarray}\label{uno}
\mathbb{E}\left(w_t\right|w_T)=\mathbb{E}^P \left(w_t\right|w_T)= \dfrac{t}{T}w_T.
\end{eqnarray}
Denoted by 
$\Phi(\omega):=\mathbb{E}\left(\cdot|w_T(\omega)\right)$, the measure $N:\mathcal{A}\rightarrow X$ defined in (\ref{due})
satisfies the following equality:
\begin{eqnarray*}
N\left(A\right)=\dint_{A}\Phi\left(\omega\right)d\mathbb{P}(\omega) =\dint_A\mathbb{E}\left(\cdot |w_T(\omega)\right)d\mathbb{P}(\omega)
\end{eqnarray*}
So, thanks to Proposition \ref{stefansson}, it is
$$\int_A\phi(\omega)dN(\omega)=\int_A \phi(\omega)\Phi(\omega) d\mathbb{P}(\omega)=\int_A \mathbb{E}\left(\phi(\omega)|w_T(\omega)\right)d\mathbb{P}(\omega).$$
\end{proof}

Using this result we are able to observe that
\begin{proposition}
Let $w_t$ be a Brownian motion and $N$ the vector measure defined in (\ref{due}) . Then for every $s < t$ it is
 $$\mathbb{E}^{N}\left(w_t|\mathcal{F}_s\right)= \dfrac{t}{T} \,w_s.$$
\end{proposition}
\begin{proof}
Since $w_t$ is a  Brownian motion  then, for every $s<t$,
and for every $A \in \mathcal{F}_s$, thanks to (\ref{uno}) and (\ref{due})
\begin{eqnarray*}
\int_A w_t dN &=& \int_A w_t \dfrac{dN}{dP}  dP = \int_A w_t  \mathbb{E}^{\mathbb{P}}\left(\cdot|w_T(\omega)\right) dP=
\int_A  \mathbb{E}^{\mathbb{P}}\left( w_t|w_T(\omega)\right) dP= \\ &=& 
\int_A  \dfrac{t}{T} w_T\, dP
\end{eqnarray*}
and so
$$\mathbb{E}^{N}\left(w_t|\mathcal{F}_s\right)=\mathbb{E}^{\mathbb{P}}\left(\mathbb{E}^{\mathbb{P}}\left(w_t|w_T\right)|\mathcal{F}_s\right)=\mathbb{E}^{\mathbb{P}}\left(w_T\frac{t}{T}|\mathcal{F}_s\right)=\dfrac{t}{T}\, \mathbb{E}^{\mathbb{P}}\left(w_T|\mathcal{F}_s\right) = \dfrac{t}{T} \,w_s.$$
\end{proof}
We highlight that $w_t$ is not a martingale with respect to the vector measure $N$.
However, it holds:
\begin{proposition}
Given $(w_t)_{t\in[0,T]}$ be a Brownian motion on the probability filtered space \spaces, we have that $w_t$ is a martingale with respect to the vector measure 
$Q:\mathcal{A}\rightarrow X$  defined by:
$$Q(A):=\int_A\mathbb{E}^{\mathbb{P}}\left(\cdot \, | \dfrac{t}{T}w_T(\omega)\right)d\mathbb{P}(\omega).$$
\end{proposition}
\begin{proof}
We can observe that, for every $s\leq t\in [0,T]$, we have that
$$\mathbb{E}^{Q}\left(w_t|\mathcal{F}_s\right)=\mathbb{E}^{\mathbb{P}}\left(\mathbb{E}^{\mathbb{P}}\left(w_t \,|\dfrac{t}{T}w_T\right)|\mathcal{F}_s\right)=\mathbb{E}^{\mathbb{P}}\left(w_T|\mathcal{F}_s\right)=w_s,$$
where we have used the fact that $\left(w_t | \dfrac{t}{T}w_T\right)$ follows a normal distribution of parameters $\left(w_T, \left(1-\dfrac{t}{T}\right)t\right)$.
Then the Brownian motion is a martingale under the vector measure $Q$.
\end{proof}

\noindent {\bf Acknowledgment}.
This research was partially supported by Grant  “Analisi reale, teoria della misura ed approssimazione per la ricostruzione di
im\-ma\-gini”
 (2020) of GNAMPA -- INDAM (Italy) and  by University of Perugia  --  Fondo Ricerca di Base 2017.


\begin{thebibliography}{99}
\small

\bibitem{BS1} 
 E. J. Balder and A.R. Sambucini  {\it A note on strong convergence for Pettis integrable function}, Vietnam J.  Math.  {\bf 31},  N. 3 (2003), 341-347.

\bibitem{bc}
J.K.  Brooks, D. Candeloro,    {\em Weak stochastic integration in Banach spaces}.
 Atti Sem. Mat. Fis. Univ. Modena  {\bf 49} (2),  (2001),  513--522.

\bibitem{candrn} 
D. Candeloro, A. Croitoru, A. Gavrilut, A.R. Sambucini, \emph{A multivalued version of the Radon-Nikodym theorem, via the single-valued Gould integral}, Australian Journal of Mathematical Analysis and Applications, {\bf 15} (2), art. 9 pp 1-16 (2018). 

\bibitem{candeloro0} 
D. Candeloro, L. Di Piazza, K. Musial, A.R. Sambucini,  \emph{Gauge integrals and selections of weakly compact valued multifunctions}, J.M.A.A  {\bf 441} (1),  (2016),  293-308, Doi: 10.1016/j.jmaa.2016.04.009.

\bibitem{candeloro1}
D. Candeloro, L. Di Piazza, K. Musial, A.R. Sambucini, \emph{Relations among gauge and Pettis integrals for multifunctions with weakly compact convex values}, Annali di Matematica,  {\bf 197} (1) , (2018), 171-183, Doi: 10.1007/s10231-017-0674-z.

\bibitem{brow}
D. Candeloro, C.C.A. Labuschagne, V. Marraffa, A.R. Sambucini, \emph{Set-valued Brownian motion}, Ricerche di Matematica, vol 67 (2), (2018), 347-360. Doi: 10.1007/s11587-018-0372-1.

\bibitem{girsanov}
 D. Candeloro, A. R. Sambucini, \emph{A Girsanov Result Through Birkhoff Integral},  
ICCSA 2018, LNCS 10960, Doi:10.1007/978-3-319-95162-1\_47 (2018)

\bibitem{cst1}
 D. Candeloro, A. R. Sambucini, L. Trastulli, \emph{A vector Girsanov result and its applications to conditional measures via the Birkhoff integrability}, Mediterr. J. Math. (2019) 16:144, https://doi.org/10.1007/s00009-019-1431-x
 
 \bibitem{ckr1} 
  C. Cascales, V. Kadets and J. Rodríguez, \textit{The Pettis integral for multi-valued functions via single-valued ones}, J. Math. Anal. Appl. {\bf  332}, (1)  (2007),   1--10.
  
\bibitem{ckr} 
 C. Cascales, V. Kadets and J. Rodríguez, \textit{Measurable selectors and set-valued Pettis integral in non-separable Banach spaces}, J. Functional Analysis {\bf 256} (2009), 673-699.


\bibitem{cc}
 K. Cicho\'n, M. Cicho\'n,   \textit{Some Applications of Nonabsolute Integrals in the Theory of Differential Inclusions in Banach Spaces}, G.P. Curbera,G. Mockenhaupt, W.J. Ricker (Eds.), Vector Measures, Integration and Related Topics, in: Operator Theory: Advances and Applications,  vol. 201, BirHauser-Verlag, ISBN: 978-3-0346-0210-5 (2010), 115--124.


\bibitem{anca} 
A. Croitoru, C. Stamane, \emph{The general Pettis–Sugeno integral of vector multifunctions relative to a vector fuzzy multimeasure}, 
Fuzzy Sets and Systems {\bf 327},   Doi: 10.1016/j.fss.2017.07.007

\bibitem{dma}
 L. Di Piazza, V. Marraffa, \textit{Pettis integrability of fuzzy mappings with values in arbitrary Banach spaces}, 
Mathematica Slovaca {\bf 67} (6), (2017)
Doi: 10.1515/ms-2017-0057

\bibitem{dmas}
 L. Di Piazza, V. Marraffa, B. Satco, \textit{Approximating the solutions of differential inclusions driven by measures}
Annali di Matematica Pura ed Applicata, {\bf 198} (6), (2019),  2123--2140.


\bibitem{dm}
 L. Di Piazza, K. Musia{\l}, \textit{Relations among Henstock, McShane and Pettis integrals for multifunctions with compact convex values}, Monatsh. Math. {\bf 173} (4),  (2014),  459--470.
 
 \bibitem{dm-new}
 L. Di Piazza, K. Musia{\l}, \textit{Decompositions of Weakly Compact Valued Integrable Multifunctions},  Mathematics, {\bf 8} (6), (2020), 863; doi: 10.3390/math8060863. 

\bibitem{dp}
 L. Di Piazza, D. Preiss,  \textit{When do McShane and Pettis integrals coincide?},    Illinois J. Math. {\bf 47} (4), (2003),  1177--1187, ISSN: 0019-2082.

\bibitem{dps1}
 L. Di Piazza, B. Satco, \textit{A new result on impulsive differential equations involving non-absolutely
convergent integrals},  J. of Math. Anal.  Appl. \textbf{352}  (2009), 954--963.

\bibitem{FMN}
A. Fernandez, F. Mayoral, F. Naranjo, {\em Bartle–Dunford–Schwartz Integral versus
Bochner, Pettis and Dunford Integrals}, Journal of Convex Analysis, {\bf 20} (2), (2013), 339-353.

\bibitem{Fremlin0} 
D. H. Fremlin, \emph{Integration  of  vector-valued  functions}, Atti Sem. Mat. Fis. Univ. Modena, {\bf 42},  205-211,(1994)


\bibitem{ex10} G. Grimmett, D. Stirzaker, {\em Probability and Random Processes}, Oxford University Press Inc. New York, 2001.

\bibitem{GL} J.J. Grobler, C.C.A. Labuschagne, {\em
Girsanov’s theorem in vector lattices}, Positivity {\bf 23}, (2019), 1065--1099.

\bibitem{KLW}
W.C. Kuo, C.C.A. Labuschagne, B.A. Watson, {\em Conditional expectations on Riesz spaces}, J. Math. Anal. Appl.,{\bf 303}, (2005), 509-521.
\bibitem{LM1}
C.C.A. Labuschagne,  V. Marraffa, {\em Operator martingale decomposition and the Radon-Nikodym property in Banach spaces}, J. Math. Anl. Appl., {\bf 363} (2), (2010), 357--365.


\bibitem{mar1}
 V. Marraffa, \emph{Stochastic processes of vector valued Pettis and McShane integrable functions}, Folia Mathematica, 
{\bf  12} (1),   25-37.
 
\bibitem{mmodena}
K. Musia\l, {\em  A Radon-Nikodym theorem for the Bartle-Dunford-Schwartz Integral}, Atti Sem. Mat. Fis. Univ. Modena , {\bf XLI}, (1993)  227-233. 

\bibitem{mikosh}
 T. Mikosch, \emph{Elementary stochastic calculus (with finance in view)}, World Scientific Publishing Co. Pte. Ltd. (1998)

\bibitem{musial}
 K. Musia\l, \emph{Martingales of Pettis Integrable Functions},  Lecture Notes in Mathematics -Springer-Verlag- 794:324-339,  (1980).

\bibitem{mu}
 K. Musia\l,  \emph{ Topics in the theory of Pettis integration},  Rend. Istit. Mat. Univ. Trieste,  {\bf 23}, (1991) 177--262.

\bibitem{mu2011}
 K.  Musia\l, \emph{Pettis Integrability of Multifunctions
with Values in Arbitrary Banach Spaces},
  J. Convex Analysis,  {\bf 18} (3), (2011),  769--810.


\bibitem{novikov}
  A. Novikov, \emph{A certain identity for stochastic integrals}, Theory of Probability \& Its Applications, {\bf 17} (4), (1972),  717–720.

\bibitem{pascucci}
 A. Pascucci, \emph{Calcolo stocastico per la finanza}, Springer-Verlag Italia, (2008)


\bibitem{pettis}
 B.J. Pettis, \emph{On the integration in vector spaces}, Trans. Amer. Math. Soc.  {\bf 44}, (1938),  277-304.

\bibitem{RvG}
 M. Riedlea, O. van Gaansb, \emph{Stochastic integration for Lévy processes with values in Banach spaces}, Stochastic Processes and their Applications {\bf 119}, (2009), 1952-1974.

\bibitem{rodriguez1}
J. Rodríguez, \emph{Absolutely summing operators and integration of vector-valued functions}, J. Math. Anal. Appl., {\bf 316} (2), (2006), 579–600.

\bibitem{stef}
G. F. Stefansson {\em $L^1$ of a vector measure}, Le Matematiche, {\bf XLVIII} (2), (1993), 219-234.

\bibitem{T84}
 M. Talagrand, {\em Pettis integral and measure theory}.
 Mem. Amer. Math. Soc. {\bf 307}, Am. Math. Soc., Providence, R. I. (1984).
 
\bibitem{neerven}
 J.M.A.M. Van Neerven, L. Weis, \emph{Stochastic Integration of function with values in a Banach space}, Studia Mathematica, {\bf 166} (2), (2005),  131-170.

 \end{thebibliography}
 \end{document}